\documentclass[12pt]{article}
\usepackage{amsmath, amsthm, amsfonts, amssymb}
\usepackage[utf8]{inputenc}
\usepackage[all,cmtip]{xy}
\usepackage{xcolor}
\usepackage{graphicx}
\graphicspath{ {./images/} }

\theoremstyle{plain}
\newtheorem{thm}{Theorem}
\newtheorem{lemma}[thm]{Lemma}
\newtheorem{prop}[thm]{Proposition}

\theoremstyle{definition}
\newtheorem{defn}{Definition}
\newtheorem{exmp}{Example}

\newcommand{\squarething}[1]{
  \noindent
  \begin{center}
  \framebox{
    \vbox{
      \vspace{4mm}
      \hbox to 5.78in { {\Large \hfill #1  \hfill} }
      \vspace{4mm}
    }
  }
  \end{center}
  \vspace*{4mm}
}

\advance \topmargin by -\headheight
\advance \topmargin by -\headsep
\evensidemargin \oddsidemargin 
\newcommand{\CA}{\mathcal{A}}

\DeclareMathOperator{\vertex}{\mathbf{Vert}}
\DeclareMathOperator{\prevertex}{\mathbf{Prevert}}

\DeclareMathOperator{\en}{End}

\DeclareMathOperator{\res}{Res}

\DeclareMathOperator{\Id}{Id}
\DeclareMathOperator{\Int}{Int}

\newcommand{\E}{E}
\newcommand{\ov}[1]{\overline#1}

\newcommand{\D}{\Delta}

\newcommand{\Z}{\mathbb{Z}}

\newcommand{\g}{\mathfrak{g}}

\newcommand{\CC}{\mathcal{C}}
\newcommand{\CF}{\mathcal{F}}

\newcommand{\CD}{\mathcal{D}}

\newcommand{\CU}{\mathcal{U}}
\DeclareMathOperator{\CQ}{\mathbf{Dist}}

\newcommand{\Set}{\mathbf{Set}}
\newcommand{\Vect}{\mathbf{Vec}}

\newcommand{\vac}{\left|0\right>}

\begin{document}

\begin{center}
{\LARGE \bf A quantum field comonad} \par \bigskip

\renewcommand*{\thefootnote}{\fnsymbol{footnote}}
{\normalsize
	Jethro van Ekeren\footnote{\texttt{jethro@impa.br}}
}

\par \bigskip
%
%
%
%
{\footnotesize Instituto de Matem\'{a}tica Pura e Aplicada (IMPA)} \\ \footnotesize{Rio de Janeiro, Brazil}

\par \bigskip
\end{center}

\vspace*{10mm}

\noindent
\textbf{Abstract.} We encapsulate the basic notions of the theory of vertex algebras into the construction of a comonad on an appropriate category of formal distributions. Vertex algebras are recovered as coalgebras over this comonad.

\vspace*{10mm}

\section{Introduction}

The monad is a basic categorical notion, especially popular in the universal algebra and theoretical computer science communities. By definition a monad consists of a functor $M : \CC \rightarrow \CC$ from a category $\CC$ to itself, together with natural morphisms $X \rightarrow M(X)$ and $M^2(X) \rightarrow M(X)$ for all $X \in \CC$, satisfying certain identities. Monads are ubiquitous since if $L$ and $R$ form a pair of adjoint functors ($L$ left adjoint and $R$ right adjoint) then $M = R \circ L$ is an example of a monad. 
There is a notion of ``algebra'' over a monad $M$, namely an object $X \in \CC$ together with a morphism $m : M(X) \rightarrow X$ compatible with the structure of $M$ in a precise sense.

A simple example of a monad is the functor $\Set \rightarrow \Set$ sending a set $S$ to the set of all finite strings of elements of $S$. An algebra over this monad is a monoid (a set endowed with an associative binary product and a unit). Another example, a case of the general construction of monads from adjoint pairs, is $\Vect_k \rightarrow \Vect_k$ the composition of the symmetric algebra functor (from the category of vector spaces over the field $k$ to commutative $k$-algebras) with the obvious forgetful functor (from commutative algebras to vector spaces). An algebra over this monad is a commutative $k$-algebra. The theory of operads, which encompasses many familiar algebraic structures, can itself be put on a monadic footing: in \cite[Section 5.5]{LV} operads are given a definition as algebras over a suitable monad of rooted trees.

There are examples with a quite different flavour too. The ultrafilter monad $\beta : \Set \rightarrow \Set$ sends a set $S$ to the set of all ultrafilters on $S$. A $\beta$-algebra is a compact Hausdorff topological space {\cite[Proposition 5.5]{Manes}}! The functor $\Vect_k \rightarrow \Vect_k$ taking a vector space $V$ to its double dual $(V^*)^*$ is a monad. An algebra over this monad is an object of the opposite category $\Vect_k^{\text{op}}$, which is the same thing as a linearly compact topological vector space over $k$ (this result, expressed in a different terminology, goes back to {\cite[Section II.6]{Lefschetz}}).

The categorical dual notion of comonad also appears naturally. The power set $P(X)$ of a set $X$ may be made into a category by giving a morphism for each inclusion of subsets. A topology $\tau$ on $X$ yields a notion of interior of a subset, and the operation $\Int_\tau$ taking a subset to its interior defines a functor from $P(X)$ to itself which is in fact a comonad. In this example we have $\Int_\tau^2 = \Int_\tau$ (an interior always equals its own interior), so $\Int_\tau$ is an example of an idempotent comonad. For another topological example: the functor sending a pointed topological space to its universal cover defines a comonad on the category of (locally contractible, path connected) pointed topological spaces \cite[Section 5.3]{Perrone}.

The notion of vertex algebra has been introduced in \cite{Bor} in the guise of a vector space equipped with an infinite collection of bilinear products, satisfying identities analogous to the skew-symmetry and Jacobi identity of Lie algebras. The chiral algebra perspective of Beilinson and Drinfeld makes the connection with Lie algebras explicit; a chiral algebra over an algebraic curve $X$ being defined as a Lie algebra within a certain pseudotensor category of $\CD$-modules on the Ran space of $X$ \cite{BD}.

A different approach (which appears to have entered the mathematical literature in \cite{Goddard}, and developed in \cite{Li.96}, \cite{kac.book} and \cite{BK.field}) starts from a formal notion of quantum field on a vector space, subject to a form of mutual locality. Through a series of elegant manipulations, one arrives at Borcherds' notion of vertex algebra.

The purpose of this note is to point out the monadic nature (or more precisely, comonadic nature) of these constructions. Described informally, there is a comonad in which the state-field correspondence serves as the counit, and in which the comultiplication corresponds to an operation
\[
a(w) \mapsto \sum_{n \in \Z} a(w)_{(n)}(-) x^{-n-1}
\]
in which a quantum field $a(w)$ on a vector space $V$ is sent to a quantum field on an appropriate vector space of quantum fields on $V$. Locality is identified as the condition making sensible this notion of ``quantum field on a set of quantum fields''. The construction of such a comonad $\E$ is spelled out in detail below, and it is proved that vertex algebras are $\E$-coalgebras.

Many of the constructions presented here are part of now standard introductory material on vertex algebras, and in particular many results are cited from the textbook \cite{kac.book}. In the book \cite{KRR}, vector spaces were exchanged for topological algebras, and the construction of the comonad given in this note relies on some technical advantages of the latter viewpoint. In particular the enveloping topological associative algebra of a vertex algebra, described in \cite{FBZ}, plays a part in our construction.

\emph{Acknowledgements}: I gratefully acknowledge many useful discussions with R. Heluani on this and related topics, and also comments of N. Nikolov on an early draft. Supported by FAPERJ grants E-26/010.002607/2019 and E- 26/201.445/2021 and CNPq grant 402449/2021-5.

\section{Construction of the comonad}

We begin by recalling the definitions of comonad and coalgebra over a comonad.
\begin{defn}
Let $\CC$ be a category. A comonad on $\CC$ is a functor $\E : \CC \rightarrow \CC$ together with natural transformations $\varepsilon : \E \Rightarrow \Id$ (called counit) and $\D : \E \Rightarrow \E^2$ (called comultiplication) such that the following diagrams commute for all $X \in \CC$:
\[
\xymatrix{
\E(X) & \E^2(X) \ar@{->}[l]_{\varepsilon_{\E(X)}} & & \E^3(X) & \E^2(X) \ar@{->}[l]_{\Delta_{\E(X)}} \\
\E^2(X) \ar@{->}[u]^{\E(\varepsilon_X)} & \E(X) \ar@{->}[ul]_{\Id_{\E(X)}} \ar@{->}[l]^{\Delta_X} \ar@{->}[u]_{\Delta_X} & & \E^2(X) \ar@{->}[u]^{\E(\Delta_X)} & \E(X) \ar@{->}[l]^{\Delta_X} \ar@{->}[u]_{\Delta_X} \\
}
\]
These are referred to as the counit and coassociativity axioms, respectively.
\end{defn}

\begin{defn}
A coalgebra over the comonad $\E$ is an object $X \in \CC$ together with a morphism $\theta : X \rightarrow \E(X)$ such that the following diagrams commute:
\[
\xymatrix{
X & \E(X) \ar@{->}[l]_{\varepsilon_{X}} & & \E^2(X) & \E(X) \ar@{->}[l]_{\Delta_{X}} \\
 & X \ar@{->}[ul]^{\Id_{X}} \ar@{->}[u]_{\theta} & & \E(X) \ar@{->}[u]^{\E(\theta)} & X \ar@{->}[l]^{\theta} \ar@{->}[u]_{\theta} \\
}
\]
\end{defn}

The basic objects in our construction are certain formal series, which we follow \cite{KRR} in referring to as continuous distributions, with coefficients in a topological associative algebra. To dispell possible ambiguities, we record the definition of topological algebra in the sense that we shall use it. 
\begin{defn}\label{def:lenient}
A topological algebra over a field $k$ consists of a unital associative algebra $U$ equipped with a set of left ideals
\[
U_0 \supset U_1 \supset U_2 \supset \ldots,
\]
such that:
\begin{enumerate}
\item $\cap_p U_p = 0$,

\item for each $u \in U$ and $p \in \Z_+$, there exists $N \in \Z_+$ such that $U_N u \subset U_p$.

\item $U$ is complete, i.e., for any sequence $u_p \in U / U_p$ such that $u_{p+1} \equiv u_p \bmod{U_p}$ for all $p \in \Z_+$, there exists $u \in U$ such that $u_p \equiv u \bmod{U_p}$ for all $p \in \Z_+$.
\end{enumerate}
A morphism $\phi : U^1 \rightarrow U^2$ of topological algebras is a homomorphism of unital associative algebras such that, for each $p_2 \in \Z_+$ there exists $p_1 \in \Z_+$ such that $\phi(U^1_{p_1}) \subset U^2_{p_2}$.
\end{defn}
We remark that a choice of topology on the base field $k$ does not enter into this definition, so a better name might be ``linearly topologised associative algebra''.
\begin{defn}
Let $U$ be a topological algebra. A $U$-module $M$ is said to be smooth if, for all $x \in M$ there exists $p \in \Z_+$ such that $U_p x = 0$.
\end{defn}

\begin{exmp}\label{ex:example1}
Let $\g = \bigoplus_{n \in \Z} \g_n$ be a $\Z$-graded Lie algebra, consider the limit
\[
U = \lim U(\g) / U(\g)\g_{\geq p},
\]
under the system of morphisms $U(\g) / U(\g)\g_{\geq p+1} \rightarrow U(\g) / U(\g)\g_{\geq p}$. Set $U_p \subset U$ to be the kernel of the canonical morphism $U \rightarrow U(\g) / U(\g)\g_{\geq p}$. Intuitively $U$ consists of infinite sums $\sum_i u_i$ where $u_i \in U(\g)$ for each $i$, and for each $p$ all but finitely summands lie in $U_p$. The second condition of Definition \ref{def:lenient} is satisfied since $\g$ is graded, and this condition in turn ensures that the product in $U$ is well defined.
\end{exmp}

\begin{defn}
Let $U$ be a topological algebra. A continuous distribution on $U$ is a formal series
\[
a(w) = \sum_{n \in \Z} a_{(n)} w^{-n-1}, \qquad a_{(n)} \in U
\]
such that, for each $p \in \Z_+$, there exists $N \in \Z_+$ such that $a_{(n)} \in U_p$ for all $n \geq N$.
\end{defn}
The term-by-term product $a(w)b(w)$ of two continuous distributions is not in general well defined. Instead we have the normally ordered product, originating in quantum field theory:
\begin{align*}
:a(w) b(w): = a(w)_+ b(w) + b(w) a(w)_-,
\end{align*}
where
\begin{align*}
a(w)_+ = \sum_{n < 0} a_{(n)} w^{-n-1}, \qquad a(w)_- = \sum_{n \geq 0} a_{(n)} w^{-n-1},
\end{align*}
and its generalisation the ``$n^{\text{th}}$ product'' operations, given in the following definition. (The normally ordered product is recovered as $:a(w)b(w): = a(w)_{(-1)}b(w)$.)
\begin{defn}
Let $a(w)$ and $b(w)$ be continuous distributions on $U$, and let $n \in \Z$. The $n^{\text{th}}$ product of $a(w)$ and $b(w)$ is another continuous distribution on $U$, defined by the formula
\begin{align}\label{eq:nth.prod.def}
a(w)_{(n)}b(w) = \res_z \left( a(z) b(w) i_{z,w}(z-w)^n - b(w) a(z) i_{w, z}(z-w)^n \right).
\end{align}
Here $\res_z$ denotes the formal residue, i.e., extraction of the coefficient of $z^{-1}$. The expressions $i_{z, w}$ and $i_{w, z}$ denote the two distinct binomial series expansions
\begin{align*}
i_{z,w}(z-w)^n &= z^n \sum_{j \in \Z_+} \binom{n}{j} z^{-j} w^j \in k((z))((w)) \\
\text{and} \quad i_{w,z}(z-w)^n &= (-w)^n \sum_{j \in \Z_+} \binom{n}{j} w^{-j} z^j \in k((w))((z)).
\end{align*}
\end{defn}
We observe that the definition of $a(w)_{(n)}b(w)$ involves certain infinite sums of elements of $U$, which converge because of the requirement of completeness in Definition \ref{def:lenient}. The statement that one obtains a well defined continuous distribution is Proposition 14.2 of \cite{KRR}.

\begin{defn}\label{def:locality}
We say that $a(w)$ and $b(w)$ are mutually local if $(z-w)^N [a(z), b(w)] = 0$ for some sufficiently large integer $N$. In particular $a(w)_{(n)}b(w) = 0$ for all $n \geq N$ for such a local pair.
\end{defn}
This is a convenient point to state the definition of vertex algebra.
\begin{defn}\label{def:va}
A vertex algebra consists of a vector space $V$, a vector $\vac \in V$, and a collection of bilinear products $V \otimes V \rightarrow V$, denoted $a \otimes b \mapsto a_{(n)}b$ for $n \in \Z$ such that
\begin{enumerate}
\item for each pair of elements $a, b \in V$ there exists $N \in \Z_+$ such that $a_{(n)}b = 0$ for all $n \geq N$,

\item for each $a \in V$ we have $a_{(n)}\vac = 0$ for all $n \geq 0$,

\item the linear map
\[
V \rightarrow (\en{V})[[w, w^{-1}]], \qquad a \mapsto Y(a, w) = \sum_{n \in \Z} a_{(n)} w^{-n-1}
\]
satisfies
\begin{align*}
Y(\vac, w) &= I_V \\
\text{and} \quad Y(a_{(n)}b, w) &= Y(a, w)_{(n)}Y(b, w),
\end{align*}
for all $a, b \in V$. Here the $n^{\text{th}}$ product on the right hand side is as defined by the formula \eqref{eq:nth.prod.def}. The property in part (1) of this definition guarantees that this $n^{\text{th}}$ product is well defined.
\end{enumerate}
\end{defn}
This definition of vertex algebra is equivalent to various others which appear in the literature. For a discussion of equivalent definitions, see {\cite[Section 1]{DK}}. We now return to the context of continuous distributions. Mutual locality is preserved by $n^{\text{th}}$ products in the following sense.
\begin{lemma}[Dong's lemma]\label{lem:Dong}
Let $a(w)$, $b(w)$ and $c(w)$ be pairwise mutually local continuous distributions on a topological algebra $U$. Then $a(w)_{(n)}b(w)$ and $c(w)$ form a local pair.
\end{lemma}
(See Lemma 15.1 of \cite{KRR} for this version in the context of formal distributions, Proposition 3.2.7 of \cite{Li.96} for the quantum field version). Let $U$ be a topological algebra and $\CF$ a vector space of mutually local continuous distributions on $U$. The closure $\ov\CF$ of $\CF$ is defined to be the span of all iterated $n^{\text{th}}$ products of elements of $\CF$, including the distribution $\mathbf{1}$ (the unit of $U$) as the product of an empty set of distributions. It may be checked that $a(w)_{(-k-1)}\mathbf{1} = \partial_w^{(k)}a(w)$ for all $k \geq 0$, it follows that $\ov\CF$ is closed under $\partial_w$.

\begin{defn}
Let $V$ be a vector space, and $a(w) = \sum_{n \in \Z} a_{(n)} w^{-n-1}$ a formal series with coefficients $a_{(n)} \in \en(V)$. The series is said to be a quantum field if, for each $b \in V$, there exists $N \in \Z$ such that $a_{(n)}b = 0$ for all $n \geq N$.
\end{defn}

\begin{lemma}\label{lem:trans}
Let $V$ be a vector space, $\vac \in V$ a vector and $T \in \en((V)$ an endomorphism satisfying $T\vac = 0$. Let $a(w)$ be a $T$-invariant quantum field on $V$, i.e., a quantum field satisfying $[T, a(w)] = \partial_w a(w)$. Then $a(w)\vac \in V[[w]]$.
\end{lemma}
Continuing in the context of Lemma \ref{lem:trans}, and denoting by $Q$ the vector space of $T$-invariant quantum fields on $V$, we may introduce $\varepsilon_V : Q \rightarrow V$ the linear map sending $a(w)$ to $a(w)\vac|_{w=0}$. This map is known as the state-field correspondence.
\begin{prop}\label{prop:state-field}
Let $V$ be a vertex algebra, and consider $T : V \rightarrow V$ defined by $T(a) = a_{(-2)}\vac$. Then $Y(a, w)$ is a $T$-invariant quantum field on $V$ for each $a \in V$. Furthermore $\varepsilon_V \circ Y = \Id_V$. In particular
\[
\varepsilon(a(w)_{(n)}b(w)) = \varepsilon(a(w))_{(n)}\varepsilon(b(w))
\]
for $a(w)$, $b(w)$ in the image of $Y$.
\end{prop}

We are now ready to set up our category.
\begin{defn}
Let $\CQ$ denote the category whose objects are pairs $(U, \CF)$ consisting of a topological algebra $U$, and a vector space $\CF$ of mutually local continuous distributions on $U$, and whose morphisms
\[
(U_1, \CF_1) \rightarrow (U_2, \CF_2)
\]
consist of a morphism $f : U_1 \rightarrow U_2$ of topological algebras such that $f(\CF_1) \subset \ov\CF_2$.
\end{defn}

We now describe our proposed comonad $\E : \CQ \rightarrow \CQ$, at first informally, and then rigorously. Let $(U, \CF) \in \CQ$. First we observe that for each $a(w) \in \CF$ and each $n \in \Z$ we have an endomorphism $a(w)_{(n)} \in \en(\ov\CF)$. We might consider the associative algebra $A \subset \en(\ov\CF)$ generated by all such endomorphisms, and we might set, for each $p \in \Z_+$, the subspace $A_p \subset A$ to be the span of all $a(w)_{(n)}$ for $n \geq p$. We claim that the completion $A^c = \lim A / A A_p$ defines a topological algebra. In fact, as we will see below, the following commutator formula holds in $A$
\begin{align}\label{eq:comm.fla.vas}
[a(w)_{(m)}, b(w)_{(n)}] = \sum_{j \geq 0} \binom{m}{j} (a(w)_{(j)}b(w))_{(m+n-j)}
\end{align}
(the sum on the right hand side is finite due to mutual locality). It follows easily that conditions in Definition \ref{def:lenient} are met. Now let $a(w) \in \CF$ once more. We use it to define a distribution $a(w)(x)$ on $A^c$ as follows:
\[
a(w)(x) = \sum_{n \in \Z} a(w)_{(n)} x^{-n-1}.
\]
Then we would like to set $\E(U, \CF) = (A^c, \widehat{\CF})$, where $\widehat\CF = \{a(w)(x) \mid a(w) \in \ov\CF\}$.

The difficulty with this proposal appears when we try to define a counit natural transformation
\[
\varepsilon : A^c \rightarrow U
\]
sending $a(w)_{(n)}$ to $a_{(n)}$. It might happen that $U$ contains central elements, leading to $a(w)_{(n)} = 0$ while $a_{(n)} \neq 0$. To remedy this problem we replace $A^c$ by an \emph{enveloping algebra} construction, described for example in \cite{FBZ}.
\begin{defn}\label{def:mode.lie.algebra}
Let $(U, \CF) \in \CQ$. Its mode Lie algebra is the vector space
\[
\g = \frac{\ov\CF \otimes k[\xi^{\pm 1}]}{\left< a(w) \otimes \xi^n + n \, a(w) \otimes \xi^{n-1} \mid \text{$a(w) \in \ov\CF$ and $n \in \Z$}\right>},
\]
where $\left< \cdots \right>$ means linear span, equipped with the Lie bracket
\[
[a(w)_{[m]}, b(w)_{[n]}] = \sum_{j \in \Z_+} \binom{m}{j} \left( a(w)_{(j)}b(w) \right)_{[m+n-j]}.
\]
Here and below we denote the image of $a(w) \otimes \xi^n$ in the quotient $\g$ by $a(w)_{[n]}$.
\end{defn}
\begin{lemma}\label{lem:Lie.V}
The vector space $\g$ with bracket defined above is a Lie algebra.
\end{lemma}
This lemma follows immediately from the following results.
\begin{prop}[{\cite[Proposition 17.1]{KRR}}]\label{prop:v.a}
Let $(U, \CF)$ be an object of $\CQ$. The vector space $\ov\CF$, equipped with the products $a(w)_{(n)}b(w)$, and $\vac = I_{U}$, is a vertex algebra.
\end{prop}

We sketch the proof. 
\begin{proof}
For $a(w) \in \ov \CF$ the normally ordered products $a(w)_{(n)}b(w)$ are well-defined and vanish for all $n \geq N$ for some $N$, by Lemma \ref{lem:Dong} and the subsequent remarks.

Thus we have a vector space $\ov \CF$, equipped with $\vac = I_{U}$, and endomorphism $T = \partial_w$. For each $a(w) \in \CF$ the series $a(w)(x) = \sum_{n \in \Z} a(w)_{(n)} x^{-n-1}$ satisfies condition (1) of Definition \ref{def:va}, and a short calculation verifies that these series are mutually local as per Definition \ref{def:locality}. Furthermore $[T, a(x)(x)] = \partial_x a(w)(x)$, $T \vac = 0$, and the vector space $\ov\CF$ is spanned by iterated finite products of coefficients of the $a(w)(x)$ applied to $\vac$.

From the facts listed above it follows (see {\cite[Theorem 4.5]{kac.book}}) that $\CF$ possesses a vertex algebra structure, under which $Y(a(w), x) = a(w)(x)$.
\end{proof}
The equality \eqref{eq:comm.fla.vas} is now a case of the general commutator formula: if $V$ is a vertex algebra and $a, b \in V$ then
\[
[a_{(m)}, b_{(n)}] = \sum_{j \in \Z_+} \binom{m}{j} (a_{(j)}b)_{(m+n-j)}
\]
for all $m, n \in \Z$. Lemma \ref{lem:Lie.V} follows from a general construction in vertex algebra theory {\cite[Theorem 4.6]{kac.book}} (which goes back to \cite{Bor}), applied to the vertex algebra $\ov\CF$.

\begin{defn}\label{def:mode.assoc.algebra}
Let $(U, \CF) \in \CQ$ and let $\g$ be its mode Lie algebra. Denote by $\g_{\geq p}$ the vector subspace of $\g$ spanned by $a(w)_{[n]}$ for $n \geq p$. Consider the limit
\[
\CA^c = \lim U(\g) / U(\g) \g_{\geq p}
\]
as in Example \ref{ex:example1}. Introduce continuous distributions
\[
a(w)(x) = \sum_{n \in \Z} a(w)_{[n]} x^{-n-1}
\]
in $\CA^c$, one for each $a(z) \in \ov\CF$. Let $I \subset \CA^c$ be the two-sided ideal generated by all coefficients of
\begin{align}\label{eq:impose.nth}
\left(a(w)_{(n)}b(w)\right)(x) - \left( a(w)(x) \right)_{(n)}\left( b(w)(x) \right).    
\end{align}
Finally define $\CU(U, \CF) = \CA^c / I$. Then $\CU(U, \CF)$ with the collection of left ideals kernels of the canonical morphisms $\CA^c \rightarrow U(\g) \g_{\geq p}$, is a topological algebra. 
\end{defn}
We remark once more that $\CU(U, \CF)$ is a replacement for the completion of $A \subset \en(\ov\CF)$, described above, with superior formal properties. In fact there is a homomorphism
\[
\CU(U, \CF) \rightarrow \en(\ov\CF)
\]
sending $a(w)_{[n]}$ to $a(w)_{(n)}$.

We are at last ready to define a comonad $(E, \Delta, \varepsilon)$ on $\CQ$, which we will call the quantum field comonad. First we define $E$.
\begin{prop}
Let $(U, \CF) \in \CQ$ and let $\CU = \CU(U, \CF)$. For each $a(w) \in \CF$ the series
\[
a(w)(x) = \sum_{n \in \Z} a(w)_{[n]} x^{-n-1}.
\]
is a continuous distribution on $\CU$, the set of distributions
\[
\widehat\CF = \{a(w)(x) \mid a(w) \in \ov\CF\}
\]
on $\CU$ are mutually local, and the assignment $\E : \CQ \rightarrow \CQ$ defined on objects by $E(U, \CF) = (\CU, \widehat \CF)$ is a functor.
\end{prop}

\begin{proof}
We have noted already that $\CU$ is a topological algebra. By definition $a(w)_{[n]} \in \g_{\geq p}$ whenever $n \geq p$, and so the distribution $a(w)(x) = \sum a(w)_{[n]} x^{-n-1}$ on $\CU$ is continuous. We must verify that these distributions are mutually local. Indeed
\begin{align*}
[a(w)(x), b(w)(y)]
&= \sum_{m, n \in \Z} \sum_{j \in \Z_+} \binom{m}{j} (a(w)_{(j)}b(w))_{[m+n-j]} x^{-m-1} y^{-n-1} \\
&= \sum_{j \in \Z_+} \sum_{\ell \in \Z} (a(w)_{(j)}b(w))_{[\ell]} y^{-\ell-1} \sum_{m \in \Z}  \binom{m}{j} x^{-m-1} y^{m-j} \\
&= \sum_{j \in \Z_+} \frac{1}{j!} \sum_{\ell \in \Z} (a(w)_{(j)}b(w))(y) \partial_y^{j} \delta(x, y),
\end{align*}
where $\delta(x, y) = (i_{x, y}-i_{y, x})(x-y)^{-1}$ is the formal Dirac delta function. It follows that
\[
(x-y)^N [a(w)(x), b(w)(y)] = 0
\]
whenever $N$ is sufficiently large that $(z-w)^N [a(z), b(w)] = 0$.
\end{proof}

\begin{prop}\label{prop:E=E2}
There exists a natural isomorphism $\Delta : E \Rightarrow E^2$.
\end{prop}

\begin{proof}
Let $(U, \CF)$ be an object of $\CQ$, and write $(U_1, \CF_1) = E(U, \CF)$ and $(U_2, \CF_2) = E^2(U, \CF)$. The algebra $U_1$ is constructed from the vertex algebra $\ov\CF$, and similarly $U_2$ from $\ov\CF_1$, in a functorial manner. It thus suffices to establish a natural isomorphism $\alpha : \ov\CF \rightarrow \ov\CF_1$ of vertex algebras.

The morphism we seek is, of course, $\alpha : a(w) \mapsto a(w)(x)$. It is a morphism of vertex algebras because $\alpha(\vac) = \vac$, and
\[
\alpha(a(w)_{(n)}b(w)) = \alpha(a(w))_{(n)}\alpha(b(w))
\]
for all $a(w), b(w) \in \ov\CF$, which in turn follows directly from the defining relation \eqref{eq:impose.nth} in $\CU$. Since $a(w)$ can be recovered from $\varepsilon(a(w)(x)) = a(w)(x)\vac|_{x=0}$, and $\varepsilon \circ \alpha = \Id$, we see that $\alpha$ is injective. On the other hand $\alpha$ is surjective by definition of $\ov\CF_1$.
\end{proof}
As suggested by the proof of Proposition \ref{prop:E=E2}, we should like to obtain the counit of our comonad as something like the state-field correspondence $\varepsilon$ of Proposition \ref{prop:state-field}. A little additional work is required to define it, however, as the topological algebra $U$ is not itself a vertex algebra.

\begin{prop}[{\cite[Theorem 3.2.10]{Li.96}}]\label{prop:mod.prop}
Let $(U, \CF) \in \CQ$ and let $M$ be a smooth $U$-module. Then $M$ acquires the structure of a module over the vertex algebra $\ov\CF$, via
\[
Y^M(a(w), x) = a(x)|_M.
\]
\end{prop}

\begin{prop}\label{prop:phi.homo}
There exists a morphism of topological algebras $\varphi : \CU(U, \CF) \rightarrow U$ such that
\[
\varphi(a(w)_{[n]}) = a_{(n)}.
\]
This induces a natural transformation $\varepsilon : \E \Rightarrow \Id$.
\end{prop}

\begin{proof}
For each value of $p \in \Z_+$, the quotient $U / U_p$ is a smooth left $U$-module. Thus $U / U_p$ acquires the structure of a module over the vertex algebra $\ov \CF$, by Proposition \ref{prop:mod.prop}. In particular the relations $Y(a_{(n)}b, w) = Y(a, w)_{(n)}Y(b, w)$ hold in $U$ modulo $U_p$. Since by definition $\cap_p U_p = 0$ and $U$ is complete, the relations hold in $U$. 

Write $(\CU, \widehat\CF)$ for $E(U, \CF)$. The morphism $\varphi : \CU \rightarrow U$ just constructed, induces an inclusion $\varphi(\widehat\CF) \subset \ov\CF$. Thus the second claim of the Proposition statement follows.
\end{proof}

\begin{lemma}
The triple $(\E, \Delta, \varepsilon)$ is an idempotent comonad on $\CQ$.
\end{lemma}

\begin{proof}
The equality $\varepsilon_{\E(X)} \circ \Delta_X = \Id$ has already been established in the course of the proof of Proposition \ref{prop:E=E2}.

As for the equality $\E(\varepsilon_{X}) \circ \Delta_X = \Id$, since $\varepsilon_X$ acts as $a(w)_{[n]} \mapsto a_{(n)}$ and thus $a(w)(x) \mapsto a(x)$. The corresponding morphism $\E(\varepsilon_X)$ acts as $a(w)(x)_{[n]} \mapsto a(x)_{[n]}$, and satisfies the required property. 

Now we turn to the coassociativity axiom. Let $\vertex$ denote the category of vertex algebras, and let $F : \vertex \rightarrow \vertex$ be the functor sending $V$ to
\[
\{Y(a, w) \mid a \in V\}.
\]
By Proposition \ref{prop:state-field} this defines an auto-equivalence of $\vertex$. We denote by $Y : \Id \Rightarrow F$ the natural transformation defined by $Y_V(a) = Y(a, w)$.

For $X \in \CQ$ the construction of $\E(X)$ is functorial in the associated vertex algebra $\ov\CF$, and upon comparison of the definitions of the natural transformations $\Delta$ and $Y$, it is clear that coassociativity follows from commutativity of the following diagram:
\[
\xymatrix{
F^2(V) & F(V) \ar@{->}[l]_{Y_{F(V)}} \\
F(V) \ar@{->}[u]^{F(Y_V)} & V \ar@{->}[l]^{Y_V} \ar@{->}[u]_{Y_V} \\
}
\]
But this is part of the statement that $F$ is a functor, so we are done.
\end{proof}

Let $\E$ be an arbitrary comonad on a category $\CC$ and let $X \in \CC$. Then $\E(X)$ together with the component $\Delta_X : \E(X) \rightarrow \E^2(X)$ is a coalgebra over $\E$. In general $\E$-coalgebras of this form are called free $\E$-coalgebras. Now let $X$ be a coalgebra over $\E$ with structure morphism $\theta : X \rightarrow \E(X)$. It is well known that if $\E$ is idempotent then $\theta$ must be an isomorphism. (See \cite[Proposition 4.2.3]{Borceaux}, see also \cite{star.functors}.) In other words coalgebras over an idempotent comonad $\E$ are all free. From this, together with Proposition \ref{prop:v.a}, the following result now follows.
\begin{prop}
A coalgebra over the comonad $(\E, \Delta, \varepsilon)$ is a vertex algebra. More precisely if $(U, \CF) \in \CQ$ is such a coalgebra, then the vector space $\CF$ equipped with vacuum vector $\vac = \mathbf{1}$, translation operator $T = \partial_w$, and $n^{\text{th}}$-products defined by \eqref{eq:nth.prod.def}, is a vertex algebra.
\end{prop}

\section{Adjoint functors and the quantum field comonad}

Let $\E : \CC \rightarrow \CC$ be a comonad. The Eilenberg-Moore category $\CC^\E$ is defined to be the category of $\E$-coalgebras, i.e., the category whose objects are objects of $\CC$ equipped with an $\E$-coalgebra structure, and whose morphisms are morphisms of $\CC$ compatible with $\E$-coalgebra structures. The Kleisli category $\CC_\E$ is defined to be the full subcategory of $\CC^\E$ consisting of free $\E$-coalgebras.

Let us denote by $\vertex$ the category of vertex algebras. From the results of the previous section, $\CQ^\E$ is equivalent to $\vertex$, and the comonad $\E$ is recovered from a pair of adjoint functors as follows. Define the functor $C : \CQ \rightarrow \vertex$ taking $(U, \CF) \in \CQ$ to the vertex algebra $\ov\CF$. Define the functor $U : \vertex \rightarrow \CQ$ taking a vertex algebra $V$ to the pair $(U(V), \CF_V)$, where $U(V)$ is the associative enveloping algebra of $V$ constructed as in Definitions \ref{def:mode.lie.algebra} and \ref{def:mode.assoc.algebra} above, and $\CF_V$ is the set of formal distributions
\[
\{ a(w) = \sum_{n \in \Z} a_{[n]} w^{-n-1} \mid a \in V \}
\]
on $U(V)$. Then $\E = U C$.

In fact it is possible to factor $C$ through a sort of dual Rees construction in the following manner. Let $\prevertex$ denote the category of pairs $(W, \CF)$ consisting of a vector space $W$ and a set $\CF$ of mutually local quantum fields on $W$. We refer to such an object as a \emph{prevertex algebra} and denote the category of prevertex algebras as $\prevertex$. We introduce a functor
\[
R : \CQ \rightarrow \prevertex
\]
defined on objects as follows. Let $(U, \CF) \in \CQ$, then $R(U, \CF) = (W, \CF_1)$ where
\[
W = \bigoplus_{p \in \Z_+} U / U_p
\]
and $\CF_1$ is the set of fields on $W$ induced from elements of $\CF$ in the obvious way. Since $\CF$ consists of continuous distributions, $\CF_1$ consists of quantum fields. Now there is a functor $G : \prevertex \rightarrow \vertex$, taking $(W, \CF)$ to the closure $\ov\CF$, defined in essentially the same way as in the proof of Proposition \ref{prop:v.a}. As in that case, the vector space $\ov\CF$ acquires a natural structure of vertex algebra.

We thus have three categories and three functors, as follows:
\[
\xymatrix{
& \vertex \ar@{->}[dl]_{U} & \\
\CQ \ar@{->}[rr]^{R} & & \prevertex \ar@{->}[ul]_{G} \\
}
\]
The functor $C$ is recovered as $C = GR$, and the comonad $\E$ as $\E = UGR$. We remark that $GRU$ is naturally equivalent to the identity functor on $\vertex$, a fact which implies idempotency of $\E$ (and seems to be a little stronger).

The relationship between the categories $\prevertex$ and $\vertex$ is intuitively similar to the relationship between $\CQ$ and $\vertex$, without the additional complications of topological algebras and completions. However the functor $G$ is not adjoint to $RU$ as there is no natural morphism relating $RUG(X)$ and $X$ for all $X \in \prevertex$, although one of the two rigidity axioms, namely
\begin{align}\label{eq:FRUadj.1}
 \xymatrix{
 U \ar@{=>}[d]^{\Id} \ar@{=>}[r] & U \circ (FR \circ U) \ar@{->}[d]^{=} \\
 U & (U \circ FR) \circ U \ar@{=>}[l] \\
 }    
\end{align}
is satisfied, being a formal consequence of the fact that $(U, GR)$ is an adjoint pair. Nor is there a relation of adjointness between $G$ and the obvious forgetful functor $\vertex \rightarrow \prevertex$, for similar reasons. Thus, although $\prevertex$ is a natural starting point, it seems necessary to deviate from it to the category $\CQ$ in order to get a good adjoint pair and our quantum field comonad $\E$.


\end{document}